\documentclass[10pt]{amsart}
\usepackage{graphicx,amssymb,amsfonts,amsmath,amsthm,newlfont}
\usepackage{epsfig,url}

\usepackage{axodraw4j}
\usepackage{pstricks}
\usepackage{color}

\usepackage[all,2cell]{xy} \UseAllTwocells \SilentMatrices

\newtheorem{thm}{Theorem}[section]
\newtheorem{cor}[thm]{Corollary}
\newtheorem{lem}[thm]{Lemma}
\newtheorem{prop}[thm]{Proposition}
\theoremstyle{definition}
\newtheorem{defn}[thm]{Definition}

\theoremstyle{remark}
\newtheorem{rem}[thm]{Remark}
\numberwithin{equation}{section}
\newcommand{\Z}{\mathbb Z}
\newcommand{\C}{\mathbb C}

\newcommand{\R}{\mathbb R}
\newcommand{\N}{\mathbb N}
\newcommand{\Pro}{\mathbb P}

\newcommand{\gr}{\mathrm{gr}}

\font \rus= wncyr10
\newcommand{\sha}{\, \hbox{\rus x} \,}

\newcommand{\zetam}{\zeta^{ \mathfrak{m}}}

\newcommand{\Q}{\mathbb Q}
\newcommand{\Li}{\mathrm{Li}}

\newcommand{\To}{\longrightarrow}

\newcommand{\x}{\mathsf{x}}

\newcommand{\Sooo}{\mathcal{S}^0_{0,0}\,}
\newcommand{\Solo}{\mathcal{S}^0_{1,0}\,}
\newcommand{\Sool}{\mathcal{S}^0_{0,1}\,}
\newcommand{\Soll}{\mathcal{S}^0_{1,1}\,}
\newcommand{\Sllo}{\mathcal{S}^1_{1,0}\,}
\newcommand{\Slol}{\mathcal{S}^1_{0,1}\,}
\newcommand{\Sloo}{\mathcal{S}^1_{0,0}\,}
\newcommand{\Slll}{\mathcal{S}^1_{1,1}\,}
\newcommand{\Stot}{\mathcal{S}}
\newcommand{\Zooo}{\mathcal{Z}^0_{0,0}\,}
\newcommand{\Zolo}{\mathcal{Z}^0_{1,0}\,}
\newcommand{\Zool}{\mathcal{Z}^0_{0,1}\,}
\newcommand{\Zloo}{\mathcal{Z}^1_{0,0}\,}
\newcommand{\Zllo}{\mathcal{Z}^1_{1,0}\,}
\newcommand{\Zlol}{\mathcal{Z}^1_{0,1}\,}
\newcommand{\Zoll}{\mathcal{Z}^0_{1,1}\,}
\newcommand{\vooo}{V^0_{0,0}\,}
\newcommand{\volo}{V^0_{1,0}\,}
\newcommand{\vool}{V^0_{0,1}\,}
\newcommand{\vloo}{V^1_{0,0}\,}
\newcommand{\vllo}{V^1_{1,0}\,}
\newcommand{\vlol}{V^1_{0,1}\,}
\newcommand{\voll}{V^0_{1,1}\,}
\newcommand{\vlll}{V^1_{1,1}\,}
\newcommand{\ZZ}{\mathcal{Z}}
\newcommand{\Zo}{\mathcal{Z}_0}
\newcommand{\Zp}{\mathcal{Z}_{\pi}}
\newcommand{\ZH}{\mathcal{Z}_{H}}
\newcommand{\Zs}{\mathcal{Z}_{s}}
\newcommand{\Mo}{\mathcal{M}}
\newcommand{\Fo}{F}
\newcommand{\B}{\mathcal{B}}

\newcommand{\lSooo}{\hat{\mathcal{S}}^0_{0,0}\,}
\newcommand{\lSolo}{\hat{\mathcal{S}}^0_{1,0}\,}
\newcommand{\lSool}{\hat{\mathcal{S}}^0_{0,1}\,}
\newcommand{\lSoll}{\hat{\mathcal{S}}^0_{1,1}\,}
\newcommand{\lSllo}{\hat{\mathcal{S}}^1_{1,0}\,}
\newcommand{\lSlol}{\hat{\mathcal{S}}^1_{0,1}\,}
\newcommand{\lSloo}{\hat{\mathcal{S}}^1_{0,0}\,}
\newcommand{\lSlll}{\hat{\mathcal{S}}^1_{1,1}\,}
\newcommand{\lStot}{\hat{\mathcal{S}}}

\newcommand{\lZloo}{\hat{\mathcal{Z}}^1_{0,0}\,}
\newcommand{\lZllo}{\hat{\mathcal{Z}}^1_{1,0}\,}

\newcommand{\lvooo}{\hat{V}^0_{0,0}\,}
\newcommand{\lvolo}{\hat{V}^0_{1,0}\,}
\newcommand{\lvool}{\hat{V}^0_{0,1}\,}
\newcommand{\lvloo}{\hat{V}^1_{0,0}\,}
\newcommand{\lvllo}{\hat{V}^1_{1,0}\,}
\newcommand{\lvlol}{\hat{V}^1_{0,1}\,}
\newcommand{\lvoll}{\hat{V}^0_{1,1}\,}
\newcommand{\lvlll}{\hat{V}^1_{1,1}\,}

\begin{document}
\author{Francis Brown and Oliver Schnetz}
\begin{title}[The zig-zag conjecture]{Proof of the zig-zag conjecture}\end{title}
\begin{abstract}   A long-standing conjecture in quantum field theory due to Broadhurst and Kreimer  states that the amplitudes of the zig-zag graphs are a certain explicit  rational multiple of 
the odd values of the Riemann zeta function. In this paper we prove this conjecture by constructing a  certain family of single-valued  multiple polylogarithms.
The zig-zag graphs  therefore provide  the only  infinite family of primitive graphs in $\phi^4_4$ theory  (in fact, in any renormalisable quantum field theory in four dimensions) whose amplitudes are now known.
\end{abstract}
\maketitle

\hfill{To David Broadhurst, a pioneer, for his 65th birthday}

\section{Introduction}
In 1995 Broadhurst and Kreimer \cite{BK} conjectured a formula for the Feynman amplitudes of a well-known family of  graphs
called the zig-zag graphs. We give a proof of this conjecture using the second author's theory of graphical functions \cite{Graphical} (see also \cite{Drummond})
and  a variant of the first author's theory of single-valued multiple  polylogarithms \cite{BrSVP}. 
The proof makes use of   a recent  theorem due to Zagier  \cite{Zagier, Li} on the evaluation of the multiple zeta values $\zeta(2,\ldots, 2, 3, 2, \ldots, 2)$
in terms of the numbers $\zeta(2m+1) \pi^{2k}$.

\subsection{Statement of the theorem} For $n\geq 3$,  let $Z_n$ denote the zig-zag graph with $n$ loops (and zero external momenta), pictured below.

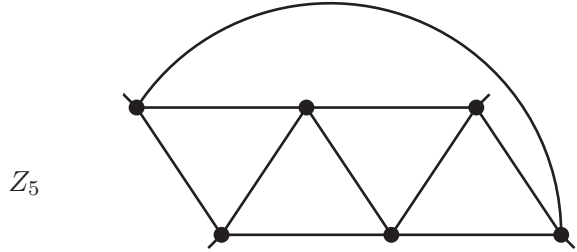
\begin{center}
\fcolorbox{white}{white}{
  \begin{picture}(194,113) (111,-32)
    \SetWidth{1.0}
    \SetColor{Black}
    \Vertex(128,33){3}
    \Vertex(192,33){3}
    \Vertex(256,33){3}
    \Vertex(160,-15){3}
    \Vertex(224,-15){3}
    \Vertex(288,-15){3}
    \Line[arrowpos=0.5,arrowlength=5,arrowwidth=2,arrowinset=0.2](128,33)(256,33)
    \Line[arrowpos=0.5,arrowlength=5,arrowwidth=2,arrowinset=0.2](128,33)(160,-15)
    \Line[arrowpos=0.5,arrowlength=5,arrowwidth=2,arrowinset=0.2](160,-15)(288,-15)
    \Line[arrowpos=0.5,arrowlength=5,arrowwidth=2,arrowinset=0.2](256,33)(288,-15)
    \Line[arrowpos=0.5,arrowlength=5,arrowwidth=2,arrowinset=0.2](160,-15)(192,33)
    \Line[arrowpos=0.5,arrowlength=5,arrowwidth=2,arrowinset=0.2](192,33)(224,-15)
    \Line[arrowpos=0.5,arrowlength=5,arrowwidth=2,arrowinset=0.2](224,-15)(256,33)
    \Arc[clock](200.855,-14.817)(87.145,146.722,-0.121)
    \Line(123,38)(128,33)
    \Line(160,-15)(155,-20)
    \Line[arrowpos=0.5,arrowlength=5,arrowwidth=2,arrowinset=0.2](288,-15)(293,-20)
    \Line[arrowpos=0.5,arrowlength=5,arrowwidth=2,arrowinset=0.2](256,33)(261,38)
   \Text(80,0)[lb]{{\Black{\large{$Z_5$}}}}
  \end{picture}
}
\end{center}

Its amplitude (scheme independent contribution to the beta function in $\phi_4^4$ theory), which is  a period in the sense of \cite{KZ}, 
can be written  in parametric space   as follows. Number the edges of $Z_n$ from $1$ to $2n$, and  to each edge  $e$  associate a variable $\alpha_e$.  The amplitude (or `period') of $Z_n$ is given by
the   convergent  integral  in projective space \cite{Wein}:
\begin{equation} \label{IZdef}
I_{Z_n} = \int_{\Delta}  {\Omega_{2n-1} \over \Psi_{Z_n} ^2 }  \in \R\ ,
\end{equation}
where $\Delta = \{ (\alpha_1: \ldots : \alpha_{2n}): \alpha_i \geq 0\} \subset \Pro^{2n-1}(\R)$ is the standard coordinate simplex, 
$$
\Omega_{2n-1} = \sum_{i=1}^{2n} (-1)^i \alpha_i \, {\mathrm d}\alpha_1\wedge \ldots \widehat{ {\mathrm d} \alpha_i} \ldots \wedge{\mathrm d} \alpha_{2n} \ ,
$$
and $\Psi_{Z_n} \in \Z[ \alpha_1, \ldots, \alpha_{2n}]$ is the graph, or Kirchhoff \cite{KIR}, 
polynomial of $Z_n$. It is  defined more generally for any graph $G$ by the formula
$$
\Psi_G = \sum_{ T \subset G} \prod_{e \notin T} \alpha_e\ ,
$$
where the sum is over all spanning trees $T$ of $G$. Since the degree of $\Psi_{Z_n}$ is equal to $n$, it follows that the integrand of $(\ref{IZdef})$ is a homogeneous $2n-1$-form  on the complement 
of the graph hypersurface $V(\Psi_G) $ in $\Pro^{2n-1}$. The integral $(\ref{IZdef})$ can be interpreted as the period of  a mixed Hodge structure \cite{BEK}.  
 It is known that the mixed Hodge structures of graph hypersurfaces can be extremely complicated \cite{BB,  K3, BD};  the mixed Hodge structures corresponding to the zig-zag graphs should be equivalent to the simplest non-trivial situation, namely  an extension of 
 $\Q(3-2n)$ by $\Q(0)$.

For $n=3, 4$ the zig-zag graphs  $Z_n$ are isomorphic to the wheels with $n$ spokes  $W_n$, whose periods are known for all $n$ by Gegenbauer polynomial techniques \cite{B2}.
For $n\geq 5$, the graphs $W_n$ are unphysical, and different from the  $Z_n$. The period for $Z_5$ was computed by Kazakov in 1983  \cite{Kazakov}, for $Z_6$ by  Broadhurst in 1985 \cite{B1} 
(see also  \cite{Ussy}), and the cases $Z_n$ for $n \leq 12$ can now be obtained by computer \cite{Graphical} using single-valued multiple polylogarithms \cite{BrSVP}. 
The period of $Z_n$ is  \emph{a priori}  known to be a multiple zeta value of weight $2n-3$ either by this method, or by the general method of  parametric integration  of \cite{BrFeyn}.
The  precise formula for its period was conjectured in \cite{BK}.

\begin{thm} \label{mainthm} (Zig-zag conjecture \cite{BK}). The period of the graph $Z_n$ is given by
\begin{equation} \label{IZ}
I_{Z_n} =  4 {(2n-2)! \over n! (n-1)!} \Big( 1 - {1- (-1)^n \over 2^{2n-3}}\Big) \zeta(2n-3)\ .
\end{equation}
\end{thm}
Using the well-known fact that the period of  a two-vertex join of a family of graphs  is the product of their periods, we immediately deduce: 
\begin{cor} \label{oddcor} Any product of  odd zeta values $\prod_{i=1}^N \zeta(2n_i+1)$, for $n_i\geq 1$,  occurs  as  the period of a primitive logarithmically-divergent graph in   $\phi_4^4$ theory.
\end{cor}

The strategy of our proof is to compute the amplitude  of the zig-zag graphs in position space by direct integration. 
At each integration step,  one has to solve  a unipotent differential equation in $\partial/ \partial z$ and $\partial / \partial \overline{z}$ on $\Pro^1\backslash \{0,1,\infty\}(\C)$, whose solution is necessarily single-valued. 
Such a method was first introduced by Davidychev and Ussyukina in \cite{Ladders} for a family of ladder diagrams. The functions they  obtained are  single-valued versions of the classical polylogarithms
$\Li_n(z)= \sum_{k\geq 1 } {z^k \over k^n}.$   A  broad generalisation of this method was recently found independently  by Schnetz \cite{Graphical} and  Drummond \cite{Drummond}, and  works for a large class of  graphs. It  uses  the fact  that any unipotent differential equation on $\Pro^1\backslash \{0,1,\infty\}$ can be solved using the single-valued multiple polylogarithms constructed in \cite{BrSVP}. 
Unfortunately, the definition of these functions is complicated and not completely explicit, so the best one can presently do by this method is to prove the zig-zag conjecture modulo products of multiple zeta values
\cite{Talk}, \cite{Graphical}.  Therefore this approach  fails  to predict the most important property of the zig-zag periods, which is that they reduce to a single Riemann zeta value.
 Experimental evidence suggests that the zig-zags may be the only $\phi^4$ periods with this property \cite{SchnetzCensus}. 

In this paper we take a  different approach, and modify the construction of the single-valued polylogarithms of \cite{BrSVP} to write down a  specific family of single-valued functions which are tailor-made for the zig-zag graphs. It does not generalise to all multiple polylogarithms, although we expect that some extensions of the present method are 
possible. The construction  relies on some special properties of the Hoffman multiple zeta values $\zeta(2,\ldots, 2, 3, 2, \ldots, 2)$ and uses a  factorization of various non-commutative
generating series into a `pure odd zeta' and `pure even zeta' part.

\subsection{Two families  of single-valued multiple polylogarithms}
Most of the paper (\S3 and \S4) is devoted to 
constructing the following explicit  families of functions. Recall that $\R\langle \langle \x_0, \x_1 \rangle \rangle$ denotes the ring of 
formal power series in two non-commuting variables $\x_0$ and $\x_1$. For any element $S \in \R\langle \langle \x_0, \x_1 \rangle \rangle$, let 
$\widetilde{S}$ denote the series obtained by reversing the letters in every word which occurs in $S$. For any word $w\in \{\x_0, \x_1\}^{\times}$, let 
$L_w(z)$ denote the multiple polylogarithm in one variable, defined by  the equations
$$
{d \over dz } L_{w  \x_i } (z) = {L_w (z) \over z-i} \quad \hbox{ for all } i = 0,1,
$$
and the condition $L_w(z) \sim 0$ as $z\rightarrow 0$ for all words $w$ not of the form $\x_0^n$, and 
$L_{\x_0^n}(z) = {1 \over n!} \log^n(z)$. The $L_w(z)$ are multi-valued functions on $\Pro^1 \backslash \{0,1,\infty\}(\C)$.

\begin{defn} \label{S0def} Define a formal power series  $\Stot \in \R \langle \langle \x_0, \x_1\rangle \rangle$ by  
$$
\Stot = 1+ \Sooo + \Slol + \Sllo + \Sloo \ ,
$$
where $\Slol = \widetilde{\Sllo}$, and 
\begin{eqnarray} \label{Sdef}
\Sooo & = & -4 \sum_{n \geq 1 }    \zeta(2n+1)\,   (\x_0 \x_1)^n \x_0 \ , \\
\Sllo & = & -4  \sum_{m\geq 1, n   \geq 0 }   \binom{2m+2n}{2m} \zeta(2m+2n+1) \,  (\x_1 \x_0)^{m}  \x_0   (\x_1 \x_0)^n \ ,\nonumber  \\
\Sloo & = &   {1 \over 2 } \big( \Sooo \Sooo + \Sooo \Sllo + \Slol \Sooo  \big)\ . \nonumber 
\end{eqnarray} 
For all $w \in \{ \x_0, \x_1\}^{\times}$, let $\Stot_w$ be  the coefficient of $w$ in $\Stot$. It is either an integer multiple of an odd single  zeta value $\zeta(2n+1)$, $n\geq 1$, or an integral  linear combination of 
products of two odd 
single zeta values $\zeta(2n+1)\zeta(2m+1)$, for $m, n\geq 1$.
\end{defn}
Let  $\B^0$ denote the set of words $w\in \{ \x_0, \x_1 \}^{\times}$  which contain no subsequences of the form $\x_1 \x_1$ or $\x_0 \x_0 \x_0$, and  have at most
one subsequence of the form $\x_0\x_0$.  These properties  are clearly stable under reversing the letters in a word, or taking a subsequence. For every $w\in \B^0$,  define a series 
\begin{equation} \label{Soldef}
F_w(z) = \sum_{w = u_1 u_2 u_3}  L_{\widetilde{u}_1} (\overline{z}) \Stot_{u_2} L_{u_3} (z) \ ,
\end{equation} 
where  $\widetilde{w}$ denotes a word $w$ written in reverse order. \emph{A priori} $F_w(z)$ is a
multivalued, real analytic  function on $\Pro^1 \backslash \{0,1, \infty\}$.  In \S3 we prove the following theorem.
\begin{thm}  \label{theoremSV1} If $w \in \B^0$, the function  $F_w(z)$ is  single-valued, and  satisfies 
$$
F_w(z)=F_{\widetilde{w}}(\overline{z})\ .
$$ 
Let $i,j \in \{0,1\}$. If $\x_i w \x_j \in \B^0$, then 
\begin{equation} \label{SV1ODE} 
{\partial^2 \over \partial z \partial \overline{z}} F_{\x_i w  \x_j } (z) = {F_w(z) \over (\overline{z} - i)(z-j)} \ .
\end{equation}
\end{thm}
The second family of functions is defined as follows.
\begin{defn} \label{S1def} Define a formal power series  $\lStot \in \R \langle \langle \x_0, \x_1\rangle \rangle$ by  
$$
\lStot = 1+ \lSooo + \lSlol + \lSllo + \lSloo \ ,
$$
where $\lSooo =  \Sooo$, $ \lSlol  =  \widetilde{\lSllo} $, and 
\begin{eqnarray} \label{S1defeq}
&&\\
\lSllo & = & -4 \sum_{m\geq 0, n   \geq 1 }  (1-2^{-2n-2m}) \binom{2m+2n}{2m+1} \zeta(2m+2n+1) \,  (\x_1 \x_0)^{m}  \x_1   (\x_1 \x_0)^n \nonumber \ ,\\
\lSloo & = & {1 \over 2 } \big(\lSooo \lSllo + \lSlol \lSooo  \big)\ . \nonumber 
\end{eqnarray} 
Note that, contrary to the previous case, the coefficients of the odd single zeta values  and their products  in $\lStot $ now have large powers of $2$ in their denominators. 
\end{defn}
Now let  $\B^1$ denote the set of words $w$ obtained from $\B^0$ by interchanging $\x_0$ and $\x_1$. 
Thus words $w\in \B^1$ contain no $\x_0\x_0$, no $\x_1\x_1\x_1$ and at most one $\x_1\x_1$. 
For every $w\in \B^1$,  define a series 
\begin{equation} \label{Soldef1}
\hat{F}_w(z) = \sum_{w = u_1 u_2 u_3}  L_{\widetilde{u}_1} (\overline{z}) \lStot_{u_2} L_{u_3} (z)  \ .
\end{equation} 
In \S4 we prove the following theorem.
\begin{thm}  \label{theoremSV2} If $w \in \B^1$, the function  $\hat{F}_w(z)$ is  single-valued, and  satisfies 
\begin{equation} 
\hat{F}_w(z) = \hat{F}_{\widetilde{w}}(\overline{z})\ . 
\end{equation}
Let $i,j \in \{0,1\}$. If $\x_i w \x_j \in \B^1$, then 
\begin{equation} \label{SV2ODE} 
{\partial^2 \over \partial z \partial \overline{z}}  \hat{F}_{\x_i w  \x_j } (z) = {\hat{F}_w(z) \over (\overline{z} - i)(z-j)} \ .
\end{equation}
\end{thm}

\begin{rem}
It is worth noting that the definition of $\Stot$ in (\ref{Sdef}) is compatible with the definition of $\lStot$ in (\ref{S1defeq}), i.e., $\Stot_w = \lStot_w$ for all 
$w\in \B^0 \cap \B^1$.
This means that it is possible to  combine the previous theorems into a single generating series
$$
F_w(z) = \sum_{w = u_1 u_2 u_3}  L_{\widetilde{u}_1} (\overline{z}) \mathcal{T}_{u_2} L_{u_3} (z)\  ,
$$
for $ w\in \B^0 \cup \B^1$,  where $\mathcal{T} = \Stot + \lStot - 1 - \Sooo$. However, this ansatz does not give single-valued functions  for all words $w$ in $\{\x_0,\x_1\}^\times$.
Since the two previous theorems are rather different in character, and play completely different roles in the zig-zag conjecture, we decided to keep their statement and proofs separate.
\end{rem}

The zig-zag conjecture itself is proved in \S5 using the two previous theorems to deal with the case when $n$ is even, or odd, respectively.
\subsection{Some remarks}
All available data suggests \cite{SchnetzCensus} that the zig-zag graphs play the same role in $\phi^4$ theory as the odd single zeta values $\zeta(2n+1)$ in the theory of multiple zeta values. The latter 
are the  primitive elements  in the algebra of motivic multiple zeta values  and correspond to the generators of the Lie algebra of the motivic Galois group
of mixed Tate motives over $\Z$.  Corollary \ref{oddcor} has the important consequence that  the periods of $\phi^4$ theory  are closed under the action of the motivic Galois group, to all known weights \cite{motivicphi4}.
It would therefore be very interesting to prove by motivic methods that $I_{Z_n}$ is an  odd  single zeta value. The fact that the multiple zeta values $\zeta(2,\ldots, 2, 3, 2,\ldots, 2)$ appear in the analytic calculation
may give some insight into a long hoped-for formula for the  motivic coaction on $\phi^4$ amplitudes.
At present, this is out of reach, but a small step in this direction was taken in \cite{K3}, where
we gave an  explicit formula for the  class of the zig-zag graph hypersurfaces $V(\Psi_{Zn}) \subset \Pro^{2n-1}$ in the Grothendieck ring of varieties. It is a polynomial in the Lefschetz motive.

Note  that it is known by the work of Rivoal and Ball-Rivoal  \cite{B-R} that the odd zeta values span an infinite dimensional vector space over the field of rational numbers.
Thus the same conclusion holds for the periods of primitive graphs in $\phi^4$ theory. In the early days of quantum field theory, the hope was often expressed that the periods 
would be rational numbers, so corollary \ref{oddcor} forms part of  an increasing body of evidence (see also  \cite{K3}), that this is very far from the truth.
\vskip2ex

\emph{Acknowledgements.} This paper was written in July 2012 whilst both authors were visiting scientists at Humboldt University, Berlin. We thank Dirk Kreimer for putting us in the same office, which led to the birth of this paper. Francis Brown is supported by ERC grant 257638.

\section{Preliminaries}

\subsection{Reminders on shuffle algebras and formal power series}
Let $R$ be a commutative unitary ring. The shuffle algebra $R\langle \x_0, \x_1 \rangle$ on two letters is the  free  $R$-module spanned by 
all words $w$ in the letters $\x_0, \x_1$, together with the empty word $1$. The shuffle product is defined recursively by  $w \sha 1 = 1 \sha w = w$ and 
$$
\x_i w \sha \x_j w'  = \x_i ( w \sha \x_j w') + \x_j ( \x_i w \sha w')
$$
for all  $i, j \in \{0,1\}$, and $w, w' \in \{ \x_0, \x_1\}^{\times}$. The shuffle product, extended linearly, makes $R\langle \x_0, \x_1\rangle$ into a
commutative unitary ring. The deconcatenation coproduct is defined to be  the linear map
\begin{eqnarray}
\Delta : R\langle \x_0, \x_1\rangle &  \To & R \langle \x_0, \x_1 \rangle \otimes_R R \langle \x_0, \x_1 \rangle \nonumber \\
\Delta (w)  &=  &\sum_{uv =w} u \otimes v \nonumber 
\end{eqnarray}
and the antipode is the linear map defined by $w \mapsto (-1)^{|w|} \widetilde{w}$, 
where $|w|\in \N$ denotes the length of a word $w$ which 
defines a grading on $R\langle \x_0, \x_1 \rangle$.  With these definitions, $R \langle \x_0, \x_1\rangle$ is a commutative, graded Hopf algebra over $R$.

The dual of $R\langle \x_0, \x_1 \rangle$ is the $R$-module of non-commutative formal power series
$$
R\langle \langle \x_0, \x_1 \rangle \rangle = \{ S= \sum_{w \in \{ \x_0, \x_1\}^{\times}} S_w w\quad , \quad  S_w \in R \}
$$
equipped with the concatenation product.  It is the completion of $R\langle \x_0, \x_1 \rangle$ 
with respect to the augmentation ideal $\ker \varepsilon$, where $\varepsilon:  R\langle \x_0, \x_1 \rangle \rightarrow R$ is the map which projects onto the empty word.
Then $R\langle \langle \x_0, \x_1 \rangle \rangle$ is a  complete Hopf algebra with respect to  the (completed) coproduct 
$$
\Gamma : R\langle \langle \x_0, \x_1\rangle \rangle  \To  R \langle \langle \x_0, \x_1 \rangle \rangle \widehat{\otimes}_R R \langle \langle \x_0, \x_1 \rangle\rangle \nonumber \\
$$
for which the elements $\x_0, \x_1$ are primitive: $\Gamma( \x_i) = 1\otimes \x_i + \x_i \otimes 1$ for $i=0,1$.  The antipode is as before. Thus $R \langle \langle \x_0, \x_1 \rangle \rangle$ is cocommutative but not commutative.

By duality, a series $S \in R \langle \langle \x_0, \x_1 \rangle \rangle$ defines an element $S \in \mathrm{Hom}_{R-mod}(R\langle \x_0, \x_1 \rangle, R)$ as follows: 
to any word $w$ associate the coefficient $S_w$ of $w$ in $S$.

An invertible  series  $S \in R\langle \langle \x_0, \x_1 \rangle \rangle^{\times}$ (i.e., with invertible leading term $S_1$) is group-like if 
$\Gamma (S) = S \otimes S$.  Equivalently, the coefficients $S_w$ of $S$   define a homomorphism for the shuffle product:
$S_{w \sha w'} = S_w S_{w'}$ for all $w, w' \in \{\x_0, \x_1 \}^{\times}$, where $S_{\bullet}$ is extended by linearity on the left-hand side.
By the formula for the antipode, it follows that for such a series $S=S(\x_0,\x_1)$, its inverse is given by
\begin{equation} \label{Sinversion}
S( \x_0, \x_1)^{-1} = \widetilde{S} (-\x_0, -\x_1)\ .
\end{equation}

\subsection{Multiple polylogarithms in one variable}
Recall that the generating series of multiple polylogarithms  on $\Pro^1\backslash \{0, 1, \infty\}$  is denoted by 
$$
L(z) = \sum_{w \in \{ \x_0, \x_1\}^{\times} } L_w(z)w \ .
$$
It is the unique solution to the  Knizhnik-Zamolodchikov equation \cite{K-Z}
\begin{equation} \label{KZ}
{d \over dz}L (z) =  L(z) \Big( {\x_0 \over z} + { \x_1 \over z-1} \Big) \ ,
\end{equation}
which satisfies the asymptotic condition 
\begin{equation} \label{Lat0}
L(z) =  \exp( \x_0 \log(z))\,  h_0(z)
\end{equation}
for all  $z$ in the neighbourhood of the origin, where
$h_0(z)$  is a function  taking values in $\C\langle\langle \x_0, \x_1\rangle \rangle$ which is holomorphic at $0$ and  satisfies $h(0)=1$.
Note that we use the opposite convention to \cite{BrSVP}  in this paper: differentiation of $L_w(z)$ corresponds to deconcatenation of $w$ on the right. 
The series $L(z)$ is a group-like formal power series.  In particular, the polylogarithms $L_w(z)$ satisfy the shuffle product formula
\begin{equation} \label{Lshuff}
L_{w \sha w'}(z) = L_w(z ) L_{w'}(z) \hbox{ for all } w,w' \in \{\x_0, \x_1\}^{\times}\ .
\end{equation}
We  have
\begin{equation}\label{classicalpolyasL}
- L_{\x_1 \x_0^{n-1}}(z) = \Li_n(z) = \sum_{k \geq 1} {z^k \over k^n} \ ,
\end{equation}
for all $n\geq 1$, which expresses the classical polylogarithms as coefficients of $L(z)$.
 
Denote the generating series of (shuffle-regularized) multiple zeta values, or Drinfeld's associator, by  
$$\ZZ(\x_0, \x_1)= \sum_{w\in \{\x_0, \x_1\}^{\times}}   \zeta( w)\, w \in \C \langle \langle \x_0, \x_1\rangle \rangle\ .$$
It is the regularized limit of $L(z)$ at the point $z=1$. In other words, there exists a function $h_1(z)$ taking values in series $ \C\langle \langle \x_0, \x_1 \rangle \rangle$, which is holomorphic at $z=1$ where it takes the value $h(1)=1$,   such that
\begin{equation}\label{Lat1}
L(z) =  \ZZ(\x_0, \x_1)  \,  \exp(\x_1 \log(1-z)) h_1(z) \ . 
\end{equation}
The series $\ZZ(\x_0, \x_1)$ is group-like, so in particular we have
\begin{equation} \label{Zinversion}
\ZZ(\x_0, \x_1)^{-1} = \widetilde{\ZZ} (-\x_0, -\x_1) \ .
\end{equation}
When no confusion arises, we denote $\ZZ(\x_0, \x_1)$ simply by $\ZZ$. Its coefficients are 
(shuffle-regularized) iterated integrals
$$\zeta(\x_{i_1} \ldots \x_{i_n} ) = \int_{0}^1 \omega_{i_1}  \ldots  \omega_{i_n}  \hbox{ for all } i_1, \ldots, i_n \in \{0, 1\}\ ,$$  
where the differential forms are integrated starting from the left,  $\omega_0 = {dz \over z}$ and $\omega_1 = {dz \over z-1}$.

For $i \in \{0, 1\}$, let $\Mo_i$ denote analytic continuation around a path winding once around the point $i$ in the positive direction.  The operators $\Mo_i$ act on the series 
$L(z)$ and $L ( \overline{z})$, commute with multiplication, and  commute with $\partial \over \partial z$ and $\partial \over \partial \overline{z}$.

\begin{lem} \label{lemmonodromy} \cite{L-D}. The monodromy operators 
$\Mo_{0} , \Mo_1$ act as follows:
\begin{eqnarray}
\Mo_{0} L(z) & = &e^{2 \pi i \x_0} \,  L(z)\ ,   \\
\Mo_{1} L(z) & =  & \ZZ e^{2 \pi i
\x_{1}} \ZZ^{-1}   L(z)\,   \ .\nonumber 
\end{eqnarray}
\end{lem}
\begin{proof}
The formula for the monodromy at the origin follows immediately from $(\ref{Lat0})$ and the equation $\Mo_0 \log (z) = \log z + 2i \pi $. From $(\ref{Lat1})$ we obtain
\begin{multline}
\Mo_1 L(z)=   \Mo_1 \big( \ZZ    \exp(\x_1 \log(1-z)) h_1(z)\big) \\
=  \ZZ  \,  \exp(2 i \pi \x_1 )  \exp(\x_1 \log(1-z)) h_1(z) =\ZZ  \,  \exp(2 i \pi \x_1 ) \ZZ^{-1} L(z)\ . 
\end{multline}
\end{proof}

\subsection{Hoffman multiple zeta values} \label{sectHoffMZV}
We need to consider a certain family of multiple zeta values similar to those first considered by Hoffman \cite{Hoffman}.

If $n_1,\ldots, n_r\geq 1$, define the following shuffle-regularized multiple zeta value:
$$\zeta_k(n_1,\ldots, n_r) = (-1)^r\zeta(\x_0^{k}\x_1\x_0^{n_1-1}  \ldots  \x_1 \x_0^{n_r-1})\ .$$
In the non-singular case $k=0, n_r \geq 2$, it reduces to the multiple zeta value 
$$\zeta(n_1,\ldots, n_r) = \sum_{0< k_1 < k_2< \ldots < k_r} {1 \over k_1^{n_1} \ldots k_r^{n_r}} \in \R\ ,$$
and we shall drop the subscript $k$ whenever it is equal to $0$. 
Henceforth, let $2^{\{n\}}$ denote a sequence $2,\ldots, 2$ of $n$ two's. Certain families of multiple zeta values will repeatedly play a role in the sequel.
The first family  corresponds to alternating words of type $(\x_1 \x_0)^n$ and reduce to even  powers of $\pi$:
\begin{equation}\label{zeta2saspi}
\zeta(2^{\{n\}}) = {\pi^{2n} \over (2n+1)!} \ .
\end{equation}

The following identity, for words of type $\x_0 (\x_1 \x_0)^n$,  is corollary 3.9 in \cite{BrMTZ}, and is easily proved using standard relations between multiple zeta values:
\begin{equation} \label{zeta12s}
 \zeta_1(2^{\{n\}}) = 2 \sum_{i=1}^{n}  (-1)^i \zeta(2i+1) \zeta(2^{\{n-i\}})\ .
\end{equation}
Next, define for any $a,b,r\in \N$, 
\begin{equation}\label{notationAB} 
A^r_{a} =  \binom{2r}{2a+2} \  \hbox{ and  } \ B^r_{b} =\bigl(1-2^{-2r}\bigr)\binom{2r}{2b+1}\ . 
\end{equation}
The following theorem is due to Zagier \cite{Zagier}, recently reproved in \cite{Li}.
\begin{thm} \label{thmZagierthm}   Let $a,b\geq 0$. Then 
\begin{equation} \label{ZagierFormula}
\zeta(2^{\{a\}}3 2^{\{b\}})=  2\,\sum_{r=1}^{a+b+1}(-1)^r (A^r_{a}-B^r_{b})\, \zeta(2r+1)\, \zeta( 2^{\{a+b+1-r\}})\,\;. 
\end{equation}
\end{thm}
We denote the corresponding generating series by:
\begin{eqnarray} \label{Zfamily}
\Zp   & = & \sum_{n\geq 0}  (-1)^n \zeta(2^{\{n\}}) \, (\x_1 \x_0)^n \ , \\
\Zo   & = &   \sum_{n\geq 1}  (-1)^n \zeta_1(2^{\{n\}}) \, \x_0 (\x_1\x_0)^n \ , \nonumber  \\
\ZH   & = &  \sum_{m,n\geq 0} (-1)^{m+n+1} \zeta(2^{\{m\}} 3 2^{\{n\}}) \, (\x_1 \x_0)^{m+1} \x_0 (\x_1 \x_0)^n \ ,\nonumber \\
\Zs   & = &   \sum_{m,n\geq 0} (-1)^{m+n+1} \zeta_1(2^{\{m\}} 3 2^{\{n\}}) \, \x_0 (\x_1 \x_0)^{m+1} \x_0 (\x_1 \x_0)^n \ . \nonumber 
\end{eqnarray}
The coefficients of $\Zp$ are even powers of $\pi$ by  $(\ref{zeta2saspi})$, and the coefficients of 
$\Zo$ and $\ZH$ are products of odd zeta values with even powers of $\pi$ by   $(\ref{zeta12s})$ and $(\ref{ZagierFormula})$.
However, the values of the `singular' Hoffman elements
$\zeta_1(2^{\{m\}}3 2^{\{n\}})$  are not known. Luckily, 
these numbers will drop out of our proofs.

\begin{rem}It turns out that the Galois coaction 
on the corresponding motivic multiple zeta values $\zetam_1(2^{\{m\}}3 2^{\{n\}})$  can be computed explicitly using the motivic version of theorem \ref{thmZagierthm} given in (\cite{BrMTZ},
theorem 4.3), and that they have `motivic depth' at most  two.  It follows  from the method of \cite{BrDecomp}  that the numbers $\zeta_1( 2^{\{m\}}3 2^{\{n\}} )$ are completely determined
up to an unknown rational multiple of $\pi^{2m+2n+4}$. 
\end{rem}

\subsection{Duality relations} \label{sectDuality}
The automorphism $z \mapsto 1-z$ of $\Pro^1 \backslash \{0,1,\infty\}$ interchanges the two forms $\omega_0 = {dz \over z}$ and  $\omega_1 = {dz \over z-1}$, and reverses 
the canonical path from $0$ to $1$. The following well-known  `duality relation'
\begin{equation} \label{MZVduality}
\zeta( \x_{i_1} \ldots \x_{i_n})= (-1)^n\zeta( \x_{1-i_n} \ldots \x_{1-i_1}) \qquad \hbox{for all } i_1,\ldots, i_n \in \{0,1\} 
\end{equation}
follows from their interpretation as  iterated integrals. Some analogous series to $(\ref{Zfamily})$ obtained by summing over sets of  words in $\B^1$, will appear in \S\ref{sectnotV1}. By
$(\ref{MZVduality})$ their coefficients can be expressed in terms of the multiple zeta values considered above.

\section{Proof of theorem \ref{theoremSV1}}\label{SV1}

\subsection{The coalgebra  of  1-Hoffman words}

\begin{defn}
Let $I_H \subset \C\langle\langle \x_0, \x_1 \rangle\rangle$ denote the (complete)
ideal generated by 
$$ w_1 \x_1^2 w_2  \ , \quad w_1 \x_0^3 w_2 \ , \quad w_1 \x_0^2 w_2 \x_0^2 w_3$$
for all $w_1,w_2, w_3 \in \C\langle\langle \x_0, \x_1 \rangle \rangle$.
\end{defn} 
Likewise, let $H\subset \C\langle \x_0, \x_1 \rangle$ denote the subspace spanned by the set  $\B^0$ of 
words  $w$ which contain no word  $\x_1\x_1$, no word $\x_0\x_0 \x_0$  and at most a single subsequence
$\x_0\x_0$. It  has an increasing filtration $\Fo$  given by  the number of
subsequences $\x_0\x_0$ (called the `level' filtration in \cite{BrMTZ}) which
satisfies $\Fo_{-1} H =0$ and $F_1 H= H.$ Thus 
$\Fo_0 H$ is the complex vector space spanned by the empty word and  
\emph{alternating words} of the form 
\begin{equation}\label{altwords}
w = \ldots \x_1 \x_0 \x_1 \x_0 \ldots \ ,
\end{equation}
(with any initial and final letter) and $\gr^\Fo_1 H$ is isomorphic to the vector space spanned by \emph{1-Hoffman
words}  of the form
\begin{equation}\label{1Hwords}
w = \ldots  \x_1 \x_0 \x_0 \x_1 \x_0  \ldots \ ,
\end{equation}
where the letters denoted by three dots are alternating (again with any initial
and final letters).
Clearly $H$  is stable under the deconcatenation coproduct:
$$\Delta: H \rightarrow H \otimes_{\C} H\ ,$$
and the filtration is compatible with deconcatenation: $\Delta F_i H \subset
\bigoplus_{j+k = i } F_j H \otimes_{\C} F_k H$, where $i,j,k \in \{-1,0,1\}$. The
coalgebra $H$ is  dual to $\C\langle \langle \x_0, \x_1 \rangle \rangle/I_H$.

\begin{defn} Let $T \subset  \C\langle\langle \x_0, \x_1 \rangle\rangle$ denote any
non-commutative formal power series $T= \sum_{w \in \{\x_0, \x_1\}^{\times}} T_w w$.
For all $i,j \in \{0,1\}$, let $T_{i,j}$ denote  the series
\begin{equation}  \label{dec1}
T_{i,j} = T_{\x_i} \x_i \, \delta_{ij} + \sum_{ w \in \x_i \{\x_0, \x_1\}^{\times}\x_j } T_w w\ ,
\end{equation}
where the sum is  over  words beginning in $\x_i$ and ending in $\x_j$.  Thus 
$$
T= T_{1}\cdot 1 + T_{0,0} + T_{1,0} +T_{0,1} +T_{1,1}\ ,
$$
where $T_1\in\C$.
Likewise, for $k=0$ or $k=1$, let 
\begin{equation} \label{dec2}
T^k = \sum_{w \in \B_k^0} T_w w\ ,
\end{equation}
where $\B_0^0\subset \B^0$ is the set of words $(\ref{altwords})$, and $\B_1^0 \subset \B^0$ is the set of words $(\ref{1Hwords})$.
Combining $(\ref{dec1})$ and $(\ref{dec2})$ gives rise to eight series $T^{k}_{i,j}$
for all $i,j,k\in \{0,1\}$.
\end{defn}

For any series $T \subset  \C\langle\langle \x_0, \x_1 \rangle\rangle$, we have 
\begin{equation}
T \equiv   T_{1} \cdot 1 +  \sum_{0\leq i,j,k \leq 1} T^{i}_{j,k}   \pmod {I_H} \ .
\end{equation}
Let $A, B \subset \C \langle \langle \x_0, \x_1 \rangle \rangle$ be any two series.
It follows from the definition of $I_H$   that
\begin{equation} \label{ABequations}
A^{1}_{*,*}B^{1}_{*,*} \equiv A^{1}_{*,0}B^{0}_{0,*} \equiv A^{0}_{*,0}B^{1}_{0,*}\ ,
\equiv A^{*}_{*,1}B^{*}_{1,*} \equiv 0 \pmod{I_H}
\end{equation}
where a $*$ denotes any index equal to $0$ or $1$. We will often use the fact that  
\begin{equation}
T \equiv 0 \pmod{I_H} \quad  \Longleftrightarrow \quad  T_1 =0 \hbox{ and }
T^i_{j,k} =0  \quad \hbox{  for all  }  \quad i,j,k \leq 1\ .  
\end{equation}
  
The following series plays an important role. 
\begin{defn} \label{defV} Let $V =   \ZZ \x_1 \ZZ^{-1} \in \C\langle \langle \x_0, \x_1 \rangle
\rangle$ and $V_-=V(-x_0,-x_1)$.
\end{defn} 
Observe that 
$\widetilde{V} = -V_{-}$ by $(\ref{Sinversion})$.
\subsection{Solutions to $(\ref{SV1ODE})$ and their monodromy equations}
We wish to construct functions $F_w(z)$ satisfying the conditions of theorem
\ref{theoremSV1}. For this, define a generating series
$F(z) = \sum_{w \in \{\x_0, \x_1\}^{\times}} F_w(z) w$
by the ansatz 
\begin{equation} \label{FAnsatz}
F(z) = \widetilde{L}(\overline{z}) \Stot L(z)\ ,
\end{equation} 
where $\Stot \subset \C \langle \langle \x_0, \x_1 \rangle \rangle$ is a constant
series which is yet to be determined.  
It follows immediately from $(\ref{FAnsatz})$ and  equation  $(\ref{KZ})$ that 
$${\partial^2 \over \partial z\partial \overline{z} } F_{\x_i w \x_j}(z)  = 
{F_w(z) \over (\overline{z} - i)(z-j)}  $$ 
for all $i, j \in \{0,1\}$, and all words $w \in \{\x_0, \x_1\}^{\times}$.   Note
that  it is \emph{not} possible to choose $\Stot$ in such a way that $(\ref{FAnsatz})$ is single-valued in general.
However, we are only interested in the coefficients $F_w(z)$ of $F$   for words $w$ which 
satisfy the conditions of theorem \ref{theoremSV1}, i.e., those words which are 
basis elements of the coalgebra $H$.  This gives rise to a weaker set of
conditions on the series $\Stot$ modulo the ideal $I_H$, which do admit a solution.

\begin{prop}  \label{prop0}  The functions $F_w(z)$  defined by $(\ref{FAnsatz})$
are single-valued and satisfy $F_w(z) = F_{\widetilde{w}}(\overline{z})$ for every
word $w\in H$  if and only if  the series $\Stot$ satisfies
\begin{eqnarray} 
&(i)&  [\Stot, x_0] \equiv 0 \pmod{I_H} \ , \nonumber \\
&(ii)&  V_- \Stot + \Stot V \equiv 0 \pmod{I_H} \ , \nonumber \\ 
&(iii)&  \overline{\Stot} \equiv  \widetilde{\Stot}   \pmod{I_H}\ . \nonumber  
\end{eqnarray}
Equation $(i)$ implies that $\Stot^*_{1,1}= \Solo=\Sool=0.$ 
\end{prop}

\begin{proof} Since the ideal $I_H$ is the annihilator of the coalgebra $H$, it 
 is enough to find conditions on $\Stot$ so that the following equations hold
\begin{eqnarray} \label{monodconditions}
\Mo_0 F(z) & \equiv & F(z) \pmod{I_H} \ ,\\
\Mo_1 F(z) & \equiv  &F(z) \pmod{I_H} \ . \nonumber 
\end{eqnarray}
For the monodromy at $0$,  lemma $(\ref{lemmonodromy})$  yields 
\begin{equation}\label{eSe}
e^{-2 i \pi \x_0 }\Stot e^{2 i \pi \x_0 } \equiv \Stot \pmod{I_H}\ .
\end{equation}
In particular, $\Stot_{1,*}   e^{2 i \pi \x_0 } \equiv \Stot_{1,*}  \pmod{I_H}$. 
There is an invertible series 
$T\in \C\langle \langle \x_0, \x_1 \rangle \rangle$ such that
$e^{2 i \pi \x_0} -1 = \x_0 T$, so we deduce that $(\ref{eSe})$ implies that
$$ \Stot_{1, *} \, \x_0 \equiv 0  \pmod{I_H}\ , $$
which implies that  
$(\Stot \x_0 )^0_{1,0} + (\Stot \x_0)^1_{1,0} = 0$. 
Removing the final letter $\x_0$ yields the equations 
$\Solo=\Soll = \Slll  = 0.$ By symmetry, we also have $\Sool=0$. Thus the only
surviving terms in $\Stot$ are of the form
\begin{equation} \label{Ssurviving}
S= S_1\cdot 1 + \Sooo + \Sllo + \Slol + \Sloo\ ,
\end{equation}
and  so $\x_0^2 \Stot \equiv  \Stot \x_0^2 \equiv \x_0 \Stot \x_0  \equiv S_1 \x_0^2
\pmod{I_H}$ by equations $(\ref{ABequations})$. Expanding out  
equation $(\ref{eSe})$, and using the fact that $\x_0^n \in I_H$ for $n\geq 3$, we
deduce that 
$$ \x_0 \Stot \equiv \Stot \x_0 \pmod{I_H}\ .$$
Conversely, this equation clearly implies $(\ref{eSe})$, so they are equivalent.

Now consider the monodromy at $1$. Lemma \ref{lemmonodromy} and 
$(\ref{monodconditions})$  yield 
the equation
$$ \widetilde{\overline{W}} \Stot W \equiv  \Stot \pmod{I_H} \ ,  $$
where $W= \ZZ e^{2 i \pi \x_1} \ZZ^{-1}$. 
Since $\x_1^2 \in I_H$, we have $W\equiv 1 + 2 i \pi V \pmod{I_H}$,  by definition \ref{defV} and the previous
equation is equivalent to 
$$ 2 i \pi ( - \widetilde{V} \Stot + \Stot V) - (2i \pi)^2  \widetilde{V} \Stot V \equiv 0 
\pmod{I_H} \ ,  $$
Since $V^2 \equiv \ZZ^2 \x_1^2 \ZZ^{-2} \equiv 0 \pmod{I_H}$, multiplying the
previous expression  on the right by $V$ yields
$\widetilde{V} SV\equiv 0 \pmod{I_H}$, and it is equivalent to the identity 
$$  - \widetilde{V} \Stot + \Stot V \equiv 0  \pmod{I_H} \ ,  $$
which gives {\it (ii)} by the equation $\widetilde{V} = - V_{-}$ after definition \ref{defV}. 

Finally, the equivalence of $(iii)$ with the equation
$F_{\widetilde{w}}(\overline{z}) = F_w(z)$ is obvious. 
\end{proof}

We can reduce equation $(ii)$ of the previous proposition further.  

\begin{lem} \label{lem0equations} If $\Stot^*_{1,1}=\Solo=\Sool=0$ and $\Stot$ is real,
the equation $V_- \Stot + \Stot V\equiv 0 \pmod{I_H}$ is equivalent to three sets of
equations:
\begin{equation}\label{firstset0}
2\vool + \Sooo \voll \equiv 0
\end{equation}
involving only the alternating part $\Sooo$, an equation involving $\Slol,\Sllo$: 
\begin{equation} \label{secondset0}
2 \vlll - \voll \Slol +\Sllo \voll  \equiv   0 
\end{equation}
and a final set of equations involving $\Sloo$ also:
\begin{eqnarray} \label{thirdset0}
\{\vool ,\Slol\}  + \Sooo ( \vool+\vlll)  & \equiv & -\Sloo \voll \ , \\
V^{\mathrm{sing}}+  [ \Sooo, \vooo]  -\vooo \Sllo +\Slol \vooo  & \equiv & -\vool \Sloo - \Sloo \volo \ , \nonumber 
\end{eqnarray}
where $V^{\mathrm{sing}} = 2\vloo-\vlol \Sooo+\Sooo \vllo$,  $\{x,y\}=xy+yx$ and $[x,y]=xy-yx$.
\end{lem} 

\begin{proof} Decompose the four equations $ (V_- \Stot + \Stot V)_{i,j}\equiv 0$, for
$i,j \in \{0,1\}$ into their parts
of odd and even weights.  After killing terms using $(\ref{ABequations})$, this
gives eight equations, one of which vanishes,
and the remaining seven are exactly the equations listed above together with the
three equations
\begin{eqnarray*}
2\volo - \voll \Sooo   & \equiv & 0\ ,\\
\{\volo , \Sllo\}  +  ( \volo+\vlll) \Sooo & \equiv & \voll \Sloo\ ,\\
\vool \Sooo + \Sooo \volo & \equiv & 0\ .
\end{eqnarray*}
The first two of these equations follow from (\ref{firstset0}) and the first equation
in (\ref{thirdset0}), respectively,  upon reflection using $\widetilde{V}=-V_-$ and $\widetilde{\Stot}=\Stot$.
The last equation is an immediate consequence of the first one and equation\ (\ref{firstset0}).
\end{proof}

In the sequel we show that our formula for $\Sooo$ given by
equation\ (\ref{Sdef}) is compatible with 
$(\ref{firstset0})$.  The non-trivial part is to check that our explicit expression for $\Sllo$ and
$\Slol$ indeed gives a solution to   $(\ref{secondset0})$.
Finally, admitting $(\ref{secondset0})$, the first equation of $(\ref{thirdset0})$ 
defines $\Sloo$,
and it is a simple matter to verify the second equation of $(\ref{thirdset0})$.
 
\subsection{Decomposition of $V$} \label{sectnotV} It follows from
$(\ref{ABequations})$ and the inversion relation $(\ref{Zinversion})$ that 
\begin{multline}  \label{Vasprod}
V \equiv ( 1 + \Zooo + \Zolo + \Zloo + \Zllo ) \, \x_1 \\ ( 1-\Zooo + \Zool + \Zloo
- \Zlol  )  \pmod{I_H}\ .
\end{multline} 
With the notations from (\ref{Zfamily}), 
we find that 
$$\Zo   =  \Zooo,\quad\Zp   =  1+  \Zolo,\quad\ZH    =  \Zllo,\quad\Zs    =  \Zloo$$
and therefore by decomposition  $(\ref{Vasprod})$ and
$\widetilde{\Zo}=\Zo$
\begin{eqnarray}\label{voeqs}
\vooo = - \Zo \x_1 \Zo\ , &&  \vool =  \Zo \x_1 \widetilde{\Zp} \ ,  \\
\volo = - \Zp \x_1 \Zo\ , &&  \voll =  \Zp \x_1 \widetilde{\Zp} \nonumber
\end{eqnarray}
for the alternating words, and 
$$\begin{array}{lll }
\vloo =  \Zo \x_1 \widetilde{\Zs}  - \Zs \x_1 \Zo\ , &&  \vlol = - \Zo \x_1
\widetilde{\ZH} + \Zs \x_1 \widetilde{\Zp}\ ,\\
\vllo = \Zp \x_1 \widetilde{\Zs} - \ZH \x_1 \Zo\ , &&  \vlll = - \Zp \x_1
\widetilde{\ZH} + \ZH \x_1 \widetilde{\Zp}\ .
\end{array}
$$

We now proceed with the verification of the equations of lemma \ref{lem0equations}.
\subsection{Alternating words}
The first task is to separate the elements $V_{*,*}^0$ into a pure odd zeta part $\Sooo$, and 
a pure `powers of $\pi$' part $\voll$. 
\begin{lem}\label{lemZo} We have
\begin{eqnarray} \label{ZoasS}
2 \Zo =  -\Sooo \Zp=   -\widetilde{\Zp} \Sooo\ .
\end{eqnarray}
\end{lem}
\begin{proof}
By the definition $(\ref{Zfamily})$ of $\Zp$  and (\ref{Sdef}) we have
\begin{eqnarray*}
- \Sooo \Zp&=& 4\sum_{m\geq1,n\geq0}(-1)^n\zeta(2m+1)\zeta(2^{\{n\}})(\x_0\x_1)^m\x_0(\x_1\x_0)^n\\
&=&4\sum_{n=0}^\infty\sum_{m=1}^n(-1)^{n-m}\zeta(2m+1)\zeta(2^{\{n-m\}})(\x_0\x_1)^n\x_0\ .
\end{eqnarray*}
The first equation in the lemma follows immediately by  $(\ref{zeta12s})$. The second equation follows from the first
by reversing the order of the words.
\end{proof}

\begin{cor}  All four series $V_{*,*}^0$ can be reduced to the single series $\voll$:
\begin{eqnarray} \label{Vfourtoone}
4 \vooo  & =  & - \Sooo \voll \Sooo \ ,\\
2 \vool & = & - \Sooo \voll\ , \nonumber \\
2 \volo & = &   \voll \Sooo\ .\nonumber
\end{eqnarray}
In particular,  equation\ (\ref{firstset0}) holds.

\end{cor} 
\begin{proof} Immediate consequence of the formulae for $V^0$  in \S\ref{sectnotV} and lemma \ref{lemZo}. 
\end{proof}

\subsection{Singular Hoffman part}
The next task is to gather all terms involving the singular Hoffman series $\Zs$,
which fortunately drops out of the final calculation.
\begin{lem} \label{lem38}  The following identity holds:
\begin{equation}\label{sing}
2V^{\mathrm{sing}} =  \Sooo \vlll \Sooo\ .
\end{equation}
\end{lem}

\begin{proof} Rewrite the elements $V_{*,*}^1$ using the formulae in \S\ref{sectnotV}. The
left-hand side gives
$$4 ( \Zo \x_1 \widetilde{\Zs}  -   \Zs \x_1 \Zo) - 2 ( \Zs \x_1
\widetilde{\Zp}  - \Zo \x_1 \widetilde{\ZH}  ) \Sooo
+  2\,  \Sooo(  \Zp \x_1 \widetilde{\Zs} - \ZH \x_1 \Zo ) $$
which is equal to
$$
(2 \Zo  + \Sooo  \Zp)\x_1(2\widetilde{\Zs} +\widetilde{\ZH}  \Sooo) - (2\Zs +
\Sooo   \ZH )\x_1(2\Zo +\widetilde{\Zp} \Sooo)+\Sooo\vlll \Sooo.
$$
By equation $(\ref{ZoasS})$ the result follows.
\end{proof}

\subsection{Hoffman part}
The main part of the calculation is the following separation of  $\vlll$ into
pure odd zeta and pure even zeta  parts.
 
\begin{lem} The following identity holds
\begin{equation}  \label{vlllid}
2 \, \vlll =  \voll \Slol - \Sllo \voll\ .
\end{equation}
\end{lem}

\begin{proof}
With the definition
$$ Y=2\ZH+\Sllo \Zp$$
we can rewrite equation\ (\ref{vlllid}) as
$$Y\x_1\widetilde{\Zp}=\Zp\x_1\widetilde{Y}\ .$$
With definition\ \ref{Sdef} and  notation $(\ref{notationAB})$, we have
\begin{eqnarray*}
\Sllo\Zp&=&-4\sum_{m\geq 1; n,k \geq 0 } A^{m+n}_{m-1} 
\zeta(2m+2n+1)(-1)^k\zeta(2^{\{k\}})\,(\x_1 \x_0)^{m}  \x_0   (\x_1 \x_0)^{n+k}\\
&=&-4\sum_{a\geq1,b\geq0}\sum_{r=a}^{a+b}(-1)^{a+b-r}A^r_{a-1}\zeta(2r+1)\zeta(2^{\{a+b-r\}})\,(\x_1
\x_0)^a  \x_0   (\x_1 \x_0)^b\ .
\end{eqnarray*}
Equation $(\ref{ZagierFormula})$  implies that in $2\ZH+\Sllo\Zp$ one binomial cancels
$$
Y=-4\sum_{a\geq1,b\geq0}\sum_{r=b+1}^{a+b}(-1)^{a+b-r}B^r_b\zeta(2r+1)\zeta(2^{\{a+b-r\}})\,(\x_1\x_0)^a  \x_0   (\x_1 \x_0)^b\ .
$$
Right multiplication by $\x_1\widetilde{\Zp}$ gives the following expression for  $-Y\x_1\widetilde{\Zp}/4$:
\begin{eqnarray*}
&&\hspace{-28pt}\sum_{a\geq1,b\geq0}\sum_{r=b+1}^{a+b}\sum_{s=0}^\infty(-1)^{a+b-r+s}B^r_b\zeta(2r+1)
\zeta(2^{\{a+b-r\}})\zeta(2^{\{s\}})\,(\x_1 \x_0)^a  \x_0   (\x_1 \x_0)^b\x_1   (\x_0\x_1)^s\\
=&&\hspace{-9pt}\sum_{\alpha ,\beta \geq1}\sum_{\gamma=0}^{\alpha-1}\sum_{\delta=0}^{\beta-1}(-1)^{\alpha+\beta-\gamma-\delta}
B^{\gamma+\delta+1}_\delta\zeta(2\gamma+2\delta+3)\zeta(2^{\{\alpha-\gamma-1\}})\zeta(2^{\{\beta-\delta-1\}})\\
&&\hspace{2cm}(\x_1 \x_0)^{\alpha}(\x_0 \x_1)^{\beta}\ ,
\end{eqnarray*}
by the change of variables $(a,b,r,s) = (\alpha , \delta, \gamma+ \delta+1, \beta-\delta-1)$. 
The last expression is evidently invariant under letter reversal which completes the proof.
\end{proof}
 
\subsection{Monodromy at zero} To prove the triviality of the monodromy at zero we need the following lemma.

\begin{lem}\label{lemxo} $[\Stot,\x_0]=0$.
\end{lem}
\begin{proof}
From the  shape $(\ref{Ssurviving})$ of $\Stot$, we find that $[\Stot,\x_0]\equiv0\pmod {I_H}$ is equivalent to
\begin{equation}\label{xo}
\Sooo\x_0+\Slol\x_0=\x_0\Sooo+\x_0\Sllo\ .
\end{equation}
According to equation\ (\ref{Sdef}) we decompose
\begin{eqnarray}\label{Smn}
\Sooo&=&\sum_{n\geq1}\Sooo(n)\,(\x_0\x_1)^n\x_0\ ,\\
\Sllo&=&\sum_{m\geq1,n\geq0}\Sllo(m,n)\,(\x_1\x_0)^m\x_0(\x_1\x_0)^n\ ,\quad\hbox{and}\nonumber\\
\Slol&=&\sum_{m\geq0,n\geq1}\Sllo(n,m)\,(\x_0\x_1)^m\x_0(\x_0\x_1)^n\ .\nonumber
\end{eqnarray}
Projecting (\ref{xo}) onto words of the form
$(\x_0\x_1)^a\x_0\x_0(\x_1\x_0)^b$ leads to identities between the coefficients which must be verified.
The case $b=0$ leads to the identity $\Sooo(a)=\Sllo(a,0)$ for all $a\in \N$. For $a,b>0$ we obtain
$\Sllo(a,b)=\Sllo(b,a)$. Both equations hold trivially by $(\ref{Sdef})$.
The case $a=0$ holds by reflection symmetry.
\end{proof}

\subsection{Proof of single-valuedness} 
To prove property $(ii)$ of proposition \ref{prop0} we need to show that equation\ (\ref{thirdset0}) holds.
The proofs are straightforward applications of $(\ref{Vfourtoone})$, (\ref{sing}), and $(\ref{vlllid})$  to write all $V$'s in terms of 
$\voll$, and reduce to  the definition of $\Sloo$ in (\ref{Sdef}).
Property $(i)$ is lemma \ref{lemxo}, and property $(iii)$ is immediately obvious from the definition of
$\Stot$. This completes the proof of  theorem \ref{theoremSV1}.

\section{Proof of theorem \ref{theoremSV2}}

This section proves the analogue of theorem \ref{theoremSV1} where zeros and ones are interchanged.
It parallels to a large extend \S\ref{SV1}. 

\subsection{The coalgebra  of  dual 1-Hoffman words}

\begin{defn}
Let $I_{\hat{H}} \subset \C\langle\langle \x_0, \x_1 \rangle\rangle$ denote the (complete)
ideal generated by 
$$  w_1 \x_0^2 w_2 \ , \quad  w_1 \x_1^3 w_2  \ , \quad w_1 \x_1^2 w_2 \x_1^2 w_3 $$
for all $w_1,w_2, w_3 \in \C\langle\langle \x_0, \x_1 \rangle \rangle$.
\end{defn} 
Likewise, let $\hat{H}\subset \C\langle \x_0, \x_1 \rangle$ denote the subspace spanned by
words $w$ which contain no word  $\x_0\x_0$,  at most a single subsequence
$\x_1\x_1$, and no $\x_1 \x_1 \x_1$.
The filtration and the notation will be the same as in \S\ref{SV1}
except that we use hat variables for quantities that live in $\hat{H}$.

\subsection{Solutions to $(\ref{SV2ODE})$ and their monodromy equations}
We again construct functions $\hat{F}_w(z)$ by a generating series
\begin{equation} \label{F1Ansatz}
\hat{F}(z) = \widetilde{L}(\overline{z}) \hat{S} L(z)\ .
\end{equation} 
The analogue of proposition \ref{prop0} is

\begin{prop}  \label{prop1} Let $\lStot$ be real.
The functions $\hat{F}_w(z)$  defined by $(\ref{F1Ansatz})$
are single-valued and satisfy $\hat{F}_w(z) = \hat{F}_{\widetilde{w}}(\overline{z})$ for every
word $w$ in $\B^1$ 
if and only if the series $\lStot$ satisfies
\begin{eqnarray} 
&(i)&  [\lStot, x_0] \equiv 0 \pmod{I_{\hat{H}}}\ ,  \nonumber \\
&(ii)&  V_- \lStot + \lStot V \equiv 0 \pmod{I_{\hat{H}}}\ ,  \nonumber \\ 
&(iii)&  \lStot \equiv  \widetilde{\lStot}   \pmod{I_{\hat{H}}}\ . \nonumber  
\end{eqnarray}
\end{prop}

\begin{proof} Equation $(i)$ is an immediate consequence of
\begin{equation}\label{eSe1}
e^{-2 i \pi \x_0 }\Stot e^{2 i \pi \x_0 } \equiv \Stot \pmod{I_{\hat{H}}}\ .
\end{equation}
Considering the monodromy at $1$, lemma \ref{lemmonodromy} yields the equation 
$$
\widetilde{\overline{W}} \lStot W \equiv \lStot \pmod{I_{\hat{H}}} \ ,
$$
where $W\equiv1+2 i \pi V+{1\over 2}(2 i \pi)^2 V^2\pmod{I_{\hat{H}}}$ and
$\widetilde{\overline{W}}\equiv1+2 i \pi V_-+{1\over 2}(2 i \pi)^2 V_-^2\pmod{I_{\hat{H}}}$.
Multiplication on the right by $V^2$ and taking the imaginary part gives
$V_-\lStot V^2\equiv0\pmod{I_{\hat{H}}}$, since $V^3 \equiv 0 \pmod{I_{\hat{H}}}$.  Likewise $V_-^2\lStot V\equiv0\pmod{I_{\hat{H}}}$.
Expanding  and taking real and imaginary parts gives the two equations
\begin{eqnarray*}
V_-\lStot+\lStot V&\equiv&0\ ,\\
V_-(V_-\lStot+\lStot V)+(V_-\lStot+\lStot V)V&\equiv&0
\end{eqnarray*}
which are 
equivalent to $(ii)$.

The equivalence of $(iii)$ with the equation
$\hat{F}_{\widetilde{w}}(\overline{z}) = \hat{F}_w(z)$ is obvious. 
\end{proof}

We use the expansions (where the upper index counts the number of $\x_1^2$'s)
$$V=\sum_{a,b,c\in\{0,1\}}\hat{V}^a_{b,c}$$
in $\hat{H}$ to reduce equation $(ii)$ of the previous proposition further.  

\begin{lem} \label{lem1equations} If $\lSoll=\lSlll=\lSolo=\lSool=0$ and $\lStot$ is real,
the equation $V_- \lStot + \lStot V\equiv 0 \pmod{I_{\hat{H}}}$ is equivalent to the equations:
\begin{eqnarray}\label{firstset1}
 2\lvool + \lSooo \lvoll &\equiv& 0\ ,\\
 2\lvlll - \lvoll \lSlol +\lSllo \lvoll &\equiv& 0\ , \nonumber\\
\{\lvool ,\lSlol\}  + \lSooo\lvlll  & \equiv & -\lSloo \lvoll\ ,  \nonumber \\
 \hat{V}^{\mathrm{sing}}-\lvooo \lSllo +\lSlol \lvooo  & \equiv & -\lvool\lSloo - \lSloo \lvolo\ ,  \nonumber 
\end{eqnarray}
where $\hat{V}^{\mathrm{sing}} = 2\lvloo-\lvlol \lSooo+\lSooo \lvllo$.
\end{lem} 

\begin{proof} The proof follows the proof of lemma  \ref{lem0equations}.
\end{proof}

Now we show that   $(\ref{firstset1})$ is consistent  with $\lStot$, as given by equation (\ref{S1defeq}).

\subsection{Decomposition of $\hat{V}$} \label{sectnotV1}
The decomposition of $\hat{V}^*_{*,*}$ into expressions in $\mathcal{Z}$ differs slightly from the previous case in \S\ref{sectnotV}.
We encounter two new types of series
\begin{eqnarray}
\Zool  & = &-\zeta_1(1)\,  \x_0 \x_1+ \zeta_1(2,1) \,  \x_0 \x_1 \x_0 \x_1 + \ldots \ , \\
\Zoll  & = & \zeta_1(2)\,  \x_1 \x_0 \x_1- \zeta_1(2,2) \,  \x_1 \x_0 \x_1 \x_0\x_1 + \ldots \ , \nonumber\\
\lZllo  & = &  \zeta(3) \x_1 \x_1 \x_0 - \zeta(2,3) \x_1 \x_1 \x_0 \x_1 \x_0 - \zeta(3,2) \x_1 \x_0 \x_1 \x_1 \x_0 + \ldots\ , \nonumber
\end{eqnarray}
where we used the duality \S\ref{sectDuality} to relate $\Zoll$ to $\Zo$ and $\lZllo$ to $\ZH$.
A further series $\lZloo$ will only be needed in intermediate steps because it, like $\Zs$, drops out of the final calculation.

The decomposition of $\hat{V}^0=V^0$ is unchanged and given by (\ref{voeqs})
whereas the components of $\hat{V}^1$ are given by (using the fact that $\widetilde{\Zoll}=\Zoll$):
\begin{eqnarray*}
\lvloo &=&  \Zo \x_1 \widetilde{\lZloo}  - \lZloo \x_1 \Zo + \Zo\x_1\widetilde{\Zool}-\Zool\x_1\Zo\ , \\ 
\lvlol &=& -\Zo \x_1 \widetilde{\lZllo} + \lZloo \x_1 \widetilde{\Zp}-\Zo\x_1\Zoll+\Zool\x_1\widetilde{\Zp}\ ,\\
\lvllo &=& \Zp \x_1 \widetilde{\lZloo} - \lZllo \x_1 \Zo +\Zp\x_1\widetilde{\Zool}-\Zoll\x_1\Zo \ ,  \\
\lvlll &=& - \Zp \x_1\Zoll + \lZllo \x_1 \widetilde{\Zp}-\Zp\x_1\widetilde{\lZllo}+\Zoll\x_1 \widetilde{\Zp}\ .
\end{eqnarray*}

We now proceed with the verification of the equations of lemma \ref{lem1equations}.
\subsection{Alternating words}
The first equation in (\ref{firstset1}) is the same as (\ref{firstset0}). 

\subsection{Dual singular Hoffman part}
The series of singular zetas $\lZloo$ drops out of the final calculation, by the following lemma. 
\begin{lem}  The following identity holds:
\begin{equation}\label{sing1}
2\hat{V}^{\mathrm{sing}} =  \Sooo \lvlll \Sooo\ .
\end{equation}
\end{lem}

\begin{proof} The calculation follows the proof of lemma  \ref{lem38}. 
\end{proof}

\subsection{Dual Hoffman part}
Again the most complicated part of the calculation is the verification of the identity for $\vlll$ in (\ref{firstset1}).

\begin{lem} The following identity holds
\begin{equation}  \label{vlllid1}
2 \, \lvlll =  \voll \lSlol - \lSllo \voll \ .
\end{equation}
\end{lem}

\begin{proof}
With the definition
$$
\hat{Y}=2\lZllo+2\Zoll+\lSllo \Zp
$$
we can rewrite equation\ (\ref{vlllid1}) as
$$
\hat{Y}\x_1\widetilde{\Zp}=\Zp\x_1\widetilde{\hat{Y}}\ .
$$
With definition\ \ref{S1def} and notation $(\ref{notationAB})$,  we have
\begin{eqnarray*}
\lSllo\Zp&=&-4\sum_{m,k\geq 0; n \geq 1 } B^{m+n}_m\zeta(2m+2n+1)(-1)^k\zeta(2^{\{k\}})\,(\x_1 \x_0)^{m}  \x_1   (\x_1 \x_0)^{n+k}\\
&=&-4\sum_{a\geq0,b\geq1}\sum_{r=a}^{a+b}(-1)^{a+b-r}B^r_a\zeta(2r+1)\zeta(2^{\{a+b-r\}})\,(\x_1\x_0)^a  \x_1   (\x_1 \x_0)^b\ ,
\end{eqnarray*}
where we have set $\zeta(1)=0$.
Equation $(\ref{ZagierFormula})$  together with the duality transformation \S\ref{sectDuality}
implies that in $2\lZllo+\lSllo\Zp$ one binomial cancels
$$
2\lZllo+\lSllo \Zp=-4\sum_{a\geq0,b\geq1}\sum_{r=b}^{a+b}(-1)^{a+b-r}A^r_{b-1}\zeta(2r+1)\zeta(2^{\{a+b-r\}})\,(\x_1\x_0)^a  \x_1   (\x_1 \x_0)^b\ .
$$
The contribution of $2\Zoll$, after applying the duality transformation $(\ref{MZVduality})$, is given by  $(\ref{zeta12s})$
$$
2\Zoll=-4\sum_{a\geq0}\sum_{r=1}^a(-1)^{a-r}\zeta(2r+1)\zeta(2^{\{a-r\}})(\x_1\x_0)^a\x_1\ .
$$
This equals the $b=0$ term in the above sum.
Multiplication by $\x_1\widetilde{\Zp}$ yields for $-\hat{Y}\x_1\widetilde{\Zp}/4$ the expression
\begin{eqnarray*}
&&\hspace{-28pt}\sum_{a,b\geq0}\sum_{r=b}^{a+b}\sum_{s=0}^\infty(-1)^{a+b-r+s}A^r_{b-1}\zeta(2r+1)
\zeta(2^{\{a+b-r\}})\zeta(2^{\{s\}})\,(\x_1 \x_0)^a  \x_1   (\x_1 \x_0)^b\x_1   (\x_0\x_1)^s\\ 
=&&\hspace{-22.4pt}\sum_{\alpha,\beta\geq0}\sum_{\gamma=0}^{\alpha}\sum_{\delta=0}^{\beta}(-1)^{\alpha+\beta-\gamma-\delta}
A^{\gamma+\delta}_{\delta-1}\zeta(2\gamma+2\delta+1)\zeta(2^{\{\alpha-\gamma\}})\zeta(2^{\{\beta-\delta\}})\,\x_1(\x_0 \x_1)^{\alpha}(\x_1 \x_0)^{\beta} \x_1.
\end{eqnarray*}
where the change of summation variables is given by $(a,b,r,s)= (\alpha, \delta, \gamma+\delta, \beta- \delta)$.
The last expression is  invariant under letter reversal which completes the proof.
\end{proof}
 
\subsection{Monodromy at zero}
As previously, we need the following lemma.

\begin{lem}\label{lemxo1} $[\lStot,\x_0]=0$.
\end{lem}
\begin{proof}
By the general shape of $\lStot$ we find that $[\lStot,\x_0]\equiv0$ mod $I_{\hat{H}}$ 
is equivalent to
\begin{equation}\label{xoexp1}
\lSlol\x_0=\x_0\lSllo\ .
\end{equation}
According to equation\ (\ref{S1defeq}) we decompose
\begin{eqnarray}\label{lSmn}
\lSllo&=&\sum_{m\geq0,n\geq1}\lSllo(m,n)\,(\x_1\x_0)^m\x_1(\x_1\x_0)^n\ ,\quad\hbox{and}\\
\lSlol&=&\sum_{m\geq1,n\geq0}\lSllo(n,m)\,(\x_0\x_1)^m\x_1(\x_0\x_1)^n\ .\nonumber
\end{eqnarray}
Projecting (\ref{xoexp1}) onto words of the form
$(\x_0\x_1)^a(\x_1\x_0)^b$ for $a,b>0$ gives the single condition  $\lSllo(a-1,b)=\lSllo(b-1,a)$ which
can be verified in equation\ (\ref{S1defeq}).
\end{proof}

\subsection{Proof of single-valuedness} 
To prove property $(ii)$ of proposition \ref{prop1} we need to show that the last two equations in equation\ (\ref{firstset1}) hold.
The proofs are straightforward applications of $(\ref{Vfourtoone})$, (\ref{sing1}), and   $(\ref{vlllid1})$  to write all $\hat{V}$'s in terms of 
$\voll$, and reduce to  the definition of $\lSloo$ in (\ref{S1defeq}).
Property $(i)$ is lemma \ref{lemxo1}, and property $(iii)$ is obvious from the definition of
$\lStot$. 
This completes the proof of  theorem \ref{theoremSV2}.

\section{Proof of the zig-zag conjecture} 

We are now ready to prove the zig-zag theorem \ref{mainthm}.

\begin{defn}
With the notation of equations (\ref{Soldef}) and  (\ref{Soldef1}) and from \cite{Graphical} we define for alternating words
$$
w=\ldots\x_0\x_1\x_0\x_1\ldots
$$
and
$$
v=\widetilde{w}\x_0\x_1w
$$
the functions $f_{2w}$ by
\begin{equation}\label{fdef}
f_{2w}(z)=(-1)^{|w|}\left\{\begin{array}{cl}\displaystyle
F_v(z)-F_{\widetilde{v}}(z) \over \displaystyle z-\overline{z}&\hbox{if $w=\x_0u$,}\\
\raisebox{3ex}{}\displaystyle\hat{F}_v(z)-\hat{F}_{\widetilde{v}}(z) \over \displaystyle z-\overline{z}&\hbox{if $w=\x_1u$.}\end{array}\right.
\end{equation}
\end{defn}

Recall that the Bloch-Wigner dilogarithm (see e.g.\ \cite{Zagierdilog})  is the single-valued
version of the dilogarithm $\Li_2(z)$ $(\ref{classicalpolyasL})$ defined by:
\begin{equation}\label{BWdilog}
D(z)=\mathrm{Im}(\Li_2(z)+\log |z| \log(1-z))\ .
\end{equation}

\begin{prop}\label{propgraph}
The functions $f_{2w}$ are real-valued, symmetric,
$$
f_{2w}(z)=f_{2w}(\overline{z})\ ,
$$
single-valued solutions to the system of differential equations
\begin{equation}\label{diffeq}
-{1\over z-\overline{z}}{\partial^2\over\partial z \partial\overline{z}}(z-\overline{z})f_{2w\x_a}(z)
={1\over (z-a)(\overline{z}-a)}f_{2w}(z)
\end{equation}
for $a\in\{0,1\}$ with the initial condition
\begin{equation}\label{f0}
f_2(z)={4iD(z)\over z-\overline{z}}\ .
\end{equation}
\end{prop}

\begin{proof}
From theorems \ref{theoremSV1} and \ref{theoremSV2}, all statements  except the last one are obvious.
To derive equation (\ref{f0}) we first observe that
$$(z-\overline{z})f_0(z)=L_{\x_1\x_0}(\overline{z})+L_{\x_0}(\overline{z})L_{\x_1}(z)+L_{\x_0\x_1}(z)-
L_{\x_1\x_0}(z)-L_{\x_0}(z)L_{\x_1}(\overline{z})-L_{\x_0\x_1}(\overline{z})\ .$$
Then one can use the shuffle product  $(\ref{Lshuff})$ to convert this expression into the dilogarithm and logarithms  via $(\ref{classicalpolyasL})$ yielding (\ref{f0}).
\end{proof}

From the theory of graphical functions \cite{Graphical}, corollary 3.28 and equation (1.9) (see also \cite{Drummond}),
which in turn uses the existence of the single-valued multiple polylogarithms \cite{BrSVP}, we have the following general theorem:

\begin{thm}\label{thegraphicalthm}
The system of differential equations (\ref{diffeq}) with initial condition (\ref{f0}) admits a unique symmetric solution $f_{2w}^{\mathrm{SVMP}}$
of the form single-valued multiple polylogarithm in $z$ divided by $z-\overline{z}$. The periods of the zig-zag graphs (\ref{IZ}) are
\begin{equation}
I_{Z_n}=f_{2w}^{\mathrm{SVMP}}(0)
\end{equation}
if $w$ is the alternating word in $\x_0$ and $\x_1$ of length $n-2$ which ends in $\x_1$.
\end{thm}

From proposition \ref{propgraph} and uniqueness we know that
$$f_{2w}^{\mathrm{SVMP}}=f_{2w}$$
given by (\ref{fdef}). To prove the zig-zag conjecture we have to determine the (regular) value of $f_{2w}$ at $z=0$. If $g$ is a single-valued function  which vanishes at $z=0$,
then setting, for example, $z= i \varepsilon $ and applying L'H{\^o}pital's rule to compute the limit as $\varepsilon \rightarrow 0$, gives
$$
\lim_{z\rightarrow 0} {g(z)\over z-\overline{z}} = {1 \over 2} \Big( {\partial g\over \partial z}(0) -{\partial g\over \partial\overline{ z}}(0)\Big)\ .
$$
Applying this formula to  (\ref{fdef}) with  the word $v=\x_1u\x_1$, where $|u|=2n-4$, we obtain
$$
f_{2w}(0)=\left\{\begin{array}{cl}
-F_{\x_1u}(0)+F_{u\x_1}(0)&\hbox{if $n$ is even,}\\
\hat{F}_{\x_1u}(0)-\hat{F}_{u\x_1}(0)&\hbox{if $n$ is odd,}\end{array}\right.
$$
where we have used  theorems \ref{theoremSV1} and \ref{theoremSV2}, and the fact that $F_{\x_1\widetilde{u}}$ and $F_{\widetilde{u}\x_1}$ are complex conjugates
of $F_{u\x_1}$ and $F_{\x_1u}$, respectively.
By theorem \ref{thegraphicalthm} the value $f_{2w}(0)$ is well-defined. Hence we  may use the regularized value (setting $\log(0)=0$) of the multiple polylogarithms at zero
to evaluate this expression. Because the regularized value at zero of any non-constant multiple polylogarithm vanishes, (\ref{Soldef}) or (\ref{Soldef1}) give
$$
I_{Z_n}=\left\{\begin{array}{cl}
-\Stot_{\x_1u}+\Stot_{u\x_1}&\hbox{if $n$ is even,}\\
\lStot_{\x_1u}-\lStot_{u\x_1}&\hbox{if $n$ is odd.}\end{array}\right.
$$
Now, in the case of even $n$ we find that $\Stot_{\x_1u}$ and $\Stot_{u\x_1}$ are summands in $\Sllo$ and $\Slol$, respectively. With notation (\ref{Smn}) we have
$$
I_{Z_n}=-\Sllo\left({n-2\over 2},{n-2\over 2}\right)+\Sllo\left({n\over 2},{n-4\over 2}\right)
$$
which evaluates by (\ref{Sdef}) to
$$
-4\left[-\binom{2n-4}{n-2}+\binom{2n-4}{n}\right]\zeta(2n-3)=4{(2n-2)! \over n! (n-1)!}\zeta(2n-3)
$$
as in theorem \ref{mainthm}.

If $n$ is odd then $\lStot_{\x_1u}$ and $\lStot_{u\x_1}$ are summands in $\lSllo$ and $\lSlol$, respectively. With notation (\ref{lSmn}) we get
$$
I_{Z_n}=\lSllo\left({n-1\over 2},{n-3\over 2}\right)-\lSllo\left({n-3\over 2},{n-1\over 2}\right)
$$
which evaluates to
$$
4(1-2^{-2n+4}){(2n-2)! \over n! (n-1)!}\zeta(2n-3)
$$
by exactly the same calculation. This completes the proof of the zig-zag theorem.

\bibliographystyle{plain}
\bibliography{main}

\end{document}